\documentclass[11pt,reqno]{amsart}

\usepackage{amssymb, amsmath, amsthm}
\usepackage{hyperref}
\usepackage[alphabetic,lite]{amsrefs}
\usepackage{verbatim}
\usepackage{amscd}   
\usepackage[all]{xy} 
\usepackage{youngtab} 
\usepackage{young} 
\usepackage{ytableau}
\usepackage{tikz}
\usepackage{ mathrsfs }

\newcommand{\defi}[1]{{\upshape\sffamily #1}}

\renewcommand{\1}{{\bf 1}}
\renewcommand{\a}{\alpha}

\newcommand{\A}{\underline{A}}
\renewcommand{\b}{\beta}

\renewcommand{\c}{\gamma}
\newcommand{\con}{\equiv}

\newcommand{\K}{\bb{K}}
\renewcommand{\ll}{\lambda}

\newcommand{\mfa}{\mathfrak{a}}
\newcommand{\mfb}{\mathfrak{b}}
\newcommand{\mfc}{\mathfrak{c}}

\newcommand{\oo}{\otimes}

\newcommand{\s}{\sigma}
\renewcommand{\S}{\mathfrak{S}}

\renewcommand{\t}{\tau}

\newcommand{\FF}{\mathscr{F}}
\newcommand{\II}{\mathscr{I}}
\renewcommand{\SS}{\mathscr{S}}
\newcommand{\TT}{\mathscr{T}}

\newcommand{\Sym}{\operatorname{Sym}} 

\newcommand{\bb}[1]{\mathbb{#1}}
\renewcommand{\rm}[1]{\textrm{#1}}
\newcommand{\mc}[1]{\mathcal{#1}}

\newcommand{\tl}[1]{\tilde{#1}}
\newcommand{\ul}[1]{\underline{#1}}

\newcommand{\ccircle}[1]{*+<1ex>[o][F-]{#1}}
\newcommand{\ccirc}[1]{\xymatrix@1{*+<1ex>[o][F-]{#1}}}

\newtheorem{theorem}{Theorem}[section]
\newtheorem*{theorem*}{Theorem}
\newtheorem{lemma}[theorem]{Lemma}

\newtheorem{proposition}[theorem]{Proposition}

\newtheorem*{corollary*}{Corollary}

\theoremstyle{definition}

\newtheorem*{definition*}{Definition}
\newtheorem{example}[theorem]{Example}

\theoremstyle{remark}
\newtheorem{remark}[theorem]{Remark}
\newtheorem*{remark*}{Remark}

\numberwithin{equation}{section}

\begin{document}

\title{Products of Young symmetrizers and ideals in the generic tensor algebra}

\author{Claudiu Raicu}
\address{Department of Mathematics, Princeton University, Princeton, NJ 08544-1000\newline
\indent Institute of Mathematics ``Simion Stoilow'' of the Romanian Academy}
\email{craicu@math.princeton.edu}


\date{\today}


\begin{abstract} We describe a formula for computing the product of the Young symmetrizer of a Young tableau with the Young symmetrizer of a subtableau, generalizing the classical quasi-idempotence of Young symmetrizers. We derive some consequences to the structure of ideals in the generic tensor algebra and its partial symmetrizations. Instances of these generic algebras appear in the work of Sam and Snowden on twisted commutative algebras, as well as in the work of the author on the defining ideals of secant varieties of Segre--Veronese varieties, and in joint work of Oeding and the author on the defining ideals of tangential varieties of Segre--Veronese varieties.
\end{abstract}

\maketitle

\section{Introduction}\label{sec:intro}

In this paper we describe a formula for computing the product of the Young symmetrizer of a Young tableau with the Young symmetrizer of a subtableau. This is a generalization of the classical result which states that an appropriate multiple of a Young symmetrizer is idempotent, and is closely related to the formulas describing the Pieri maps in \cite[Section~5]{olver}. The main motivation for our investigation comes from the study of the equations of special varieties with an action of a product of general linear groups. The GL-modules of equations correspond via Schur-Weyl duality to certain representations of symmetric groups, which we refer to as \defi{generic equations}. Understanding the ideal structure of the generic equations depends substantially on understanding how the Young symmetrizers multiply. Special instances of our main result (Theorem~\ref{thm:main}) and its application (Theorem~\ref{thm:idgentensoralg}) are implicit in \cites{rai-GSS,oed-rai} where we establish and generalize conjectures 
of Garcia--Stillman--Sturmfels and Landsberg--Weyman on the equations of the secant and tangential varieties of a Segre variety (see also Section \ref{subsec:symmetric}). We expect that Theorem~\ref{thm:main} will have applications to a number of other problems such as
\begin{itemize}
 \item describing the relations between the minors of a generic matrix \cite{bru-con-var}.
 \item proving (weak) Noetherianity of certain twisted commutative algebras generated in degree higher than one \cite{sam-snowden}.
 \item determining the equations of secant varieties of Grassmannians.
\end{itemize}

We introduce some notation before stating our main results: see Section \ref{sec:prelim} for more details, and \cites{james,ful-har,fulton} for more on Young tableaux and the representation theory of symmetric groups. For $\ll$ a partition of $n$ (denoted $\ll\vdash n$) a \defi{Young tableau} of shape $\ll$ is a collection of boxes filled with entries $1,2,\cdots,n$, arranged in left-justified rows of lengths $\ll_1\geq\ll_2\geq\cdots$. For $\mu$ a partition with $\mu_i\leq\ll_i$ for all $i$ (denoted $\mu\subset\ll$), the subtableau $S$ of $T$ of shape $\mu$ is obtained by selecting for each $i$ the first $\mu_i$ entries in the $i$-th row of $T$. For example for $\ll=(4,3,1,1)$ and $\mu=(2,1,1)$, one can take
\[\Yvcentermath1 T=\young(1234,567,8,9)\rm{ and }S=\young(12,5,8).\]
We write $\S_n$ for the symmetric group of permutations of $\{1,2,\cdots,n\}$ and $\K[\S_n]$ for its group algebra, where $\K$ is any field of characteristic zero. To any Young tableau $T$ and subtableau $S$ as above, we associate the \defi{Young symmetrizers} $\mfc_{\ll}(T)$ and $\mfc_{\mu}(S)$, which are elements of $\K[\S_n]$ (see (\ref{eq:Youngsymm}) for a precise formula).

\begin{theorem}\label{thm:main}
 Let $k\leq n$ be positive integers, let $\ll\vdash n$, $\mu\vdash k$ be partitions with $\mu\subset\ll$, and let $T$ be a Young tableau of shape $\ll$ containing a Young subtableau $S$ of shape $\mu$. We write $L(T;S)\subset\S_n$ for the set of permutations $\s$ with the property that for every entry $s$ of $S$, either $\s(s)=s$ or $\s(s)$ lies in a column of $T$ strictly to the left of the column of $s$, and moreover if $\s(s)=s$ for all $s\in S$ then $\s=\1$ is the identity permutation. There exist $m_{\s}\in\bb{Q}$ such that
 \begin{equation}\label{eq:prodyng}
\mfc_{\ll}(T)\cdot \mfc_{\mu}(S)=\mfc_{\ll}(T)\cdot\left(\sum_{\s\in L(T;S)}m_{\s}\cdot\s\right),  
 \end{equation}
where $m_{\1}=\a_{\mu}$ is the product of the hook lengths of $\mu$ (see (\ref{eq:prodhooks})). We can take $m_{\s}\geq 0$ when $\s$ is an even permutation, and $m_{\s}\leq 0$ when $\s$ is odd.
\end{theorem}
When $V$ is a vector space, $\ll\vdash n$, and $\mfc_{\ll}$ is some Young symmetrizer, we can think of multiplication by $\mfc_{\ll}/\a_{\ll}$ on $V^{\oo n}$ as a projection $V^{\oo n}\to S_{\ll}V$, where $S_{\ll}$ denotes the Schur functor associated to $\ll$. In particular, (\ref{eq:prodyng}) describes a surjective map $S_{\mu}V\oo V^{\oo(n-k)}\to S_{\ll}V$ (equivalently, the expression (\ref{eq:prodyng}) is always non-zero), so it provides some information about the $\ll$--isotypic component of the ideal generated by an irreducible component $S_{\mu}V$ in the tensor algebra $\bigoplus_{m\geq 0}V^{\oo m}$. We formulate this more precisely for the generic tensor algebra in Theorem~\ref{thm:idgentensoralg}.

Note that in the case when $T=S$ formula (\ref{eq:prodyng}) is the classical statement that $\mfc_{\ll}(T)/\a_{\ll}$ is an idempotent of $\K[\S_n]$ (see (\ref{eq:idempyngsymm})). The conclusion that $m_{\s}\in\bb{Q}$ is sufficient for our applications, but we believe that a version of (\ref{eq:prodyng}) is valid where $m_{\s}$ are in fact integers. We prove that this is the case when $n=k+1$ in Theorem~\ref{thm:yngconsec} below, where a precise formula for the coefficients $m_{\s}$ is given. The case $n=k+1$ of Theorem~\ref{thm:main} also follows from \cite[Theorem~5.2]{olver}, which however does not guarantee the integrality of the $m_{\s}$'s. To formulate Theorem~\ref{thm:yngconsec} we need some more notation. We can divide any tableau $T$ of shape $\ll$ into rectangular blocks $B_1,B_2,\cdots$ by grouping together the columns of the same size. For example when $\ll=(4,3,1,1)$ there are three blocks $B_1,B_2,B_3$:
\begin{equation}\label{eq:blocks}
\Yvcentermath1 T=\young(1234,567,8,9),\ B_1=\young(1,5,8,9),\ B_2=\young(23,67),\ B_3=\young(4). 
\end{equation}
The \defi{lengths} and \defi{heights} of the blocks $B_i$ are defined in the obvious way: in the example above they are $1,2,1$ and $4,2,1$ respectively.

\begin{theorem}\label{thm:yngconsec}
 With the notation in Theorem \ref{thm:main}, assume that $n=k+1$ and that the unique entry in $T$ outside $S$ is equal to $a$, and is located in position $(u,v)$, i.e. $\ll_i=\mu_i$ for $i\neq u$, $\ll_u=v=\mu_u+1$. Let $\tl{S}$ be the Young tableau obtained by removing from $S$ its first $v$ columns: the shape $\tl{\mu}$ of $\tl{S}$ has $\tl{\mu}_i=\mu_i-v$ for $i<u$, and $\tl{\mu}_i=0$ otherwise. Write $\tl{B}_1,\cdots,\tl{B}_{\tl{m}}$ for the blocks of $\tl{S}$ and denote by $\tl{l}_i$ (resp. $\tl{h}_i$) the length (resp. height) of the block $\tl{B}_i$. For each $i=1,\cdots,\tl{m}$, write $\tl{r}_i$ for the length of the hook of $\mu$ centered at $(\tl{h}_i,v)$, i.e.
 \[\tl{r}_i=\tl{l}_1+\cdots+\tl{l}_i+u-\tl{h}_i.\] 
If we define the elements $\tl{x}_i\in\K[\S_n]$ by $\tl{x}_i=\sum_{b\in \tl{B}_{i}} (a,b)$, then
\begin{equation}\label{eq:prodconsec}
\mfc_{\ll}(T)\cdot \mfc_{\mu}(S)=\mfc_{\ll}(T)\cdot\a_{\mu}\cdot\prod_{i=1}^{\tl{m}}\left(1-\frac{\tl{x}_i}{\tl{r}_i}\right). 
\end{equation}
\end{theorem}
\noindent In the case when the blocks $\tl{B}_i$ consist of a single column ($\tl{l}_i=1$ for $i=1,\cdots,\tl{m}$), it follows from the discussion following Theorem~\ref{thm:main} that (\ref{eq:prodconsec}) is equivalent to the formulas describing the Pieri maps $S_{\mu}V\oo V\to S_{\ll}V$ in \cite[(5.5)]{olver}, up to a change in the convention used for constructing Young symmetrizers (see also \cite{sam}). Our approach offers an alternative to that of \cite{olver}, in that we work entirely in the group algebra of the symmetric group.

\begin{example}\label{ex:prodconsec}
 Let $T$ be as in (\ref{eq:blocks}), $n=9$, $k=8$, and let $S$ be the subtableau of $T$ obtained by removing the box $\Yvcentermath1\young(9)$, so $(u,v)=(4,1)$ and $a=9$. We have $\tl{m}=2$, $\tl{r}_1=4$, $\tl{r}_2=6$, and
 \[\tl{x}_1=(9,2)+(9,3)+(9,6)+(9,7),\ \tl{x}_2=(9,4).\]
 If we use the same $T$ and take $S$ to be the subtableau obtained by removing $\Yvcentermath1\young(7)$, then $a=7$, $(u,v)=(2,3)$, $\tl{m}=1$, $\tl{r}_1=2$, and $\tl{x}_1=(7,4)$. Finally, if $S$ is the subtableau obtained by removing $\Yvcentermath1\young(4)$ then $\tl{\mu}$ is empty and (\ref{eq:prodconsec}) becomes $\mfc_{\ll}(T)\cdot \mfc_{\mu}(S)= \a_{\mu}\cdot \mfc_{\ll}(T)$.
\end{example}

\begin{remark}\label{rem:prodconsec}
 When expanding the product formula (\ref{eq:prodconsec}), all the denominators are products of distinct hook lengths of the Young diagram $\mu$, so they divide $\a_{\mu}$. Moreover, any product of distinct $\tl{x}_i$'s is a linear combination of cyclic permutations 
 \[\s=(a,b_1)\cdot(a,b_2)\cdots(a,b_r)=(a,b_r,b_{r-1},\cdots,b_1),\]
 where $b_i$ appears to the left of $b_{i+1}$ in $T$, in particular $\s\in L(T;S)$. Furthermore, $m_{\s}\geq 0$ when $r$ is even, and $m_{\s}\leq 0$ when $r$ is odd. We get that Theorem~\ref{thm:yngconsec} implies the special case of Theorem~\ref{thm:main} when $n=k+1$.
\end{remark}

As an application of Theorem \ref{thm:main} we derive certain ideal membership relations of Young symmetrizers with respect to ideals in the generic tensor algebra. We describe a preliminary version of this algebra here in order to state the results, while in Section~\ref{sec:gentensor} we take a more functorial approach. The \defi{generic tensor algebra} is the $\K$-vector space $\TT=\bigoplus_{n\geq 0}\K[\S_n]$ with multiplication defined on the basis of permutations and extended linearly, as follows. For $\s\in\S_n$ and $\t\in\S_m$, $\s*\t\in\S_{n+m}$ is given by
\[(\s*\t)(i)=\begin{cases}
\s(i) & \rm{for }i\leq n; \\
n + \t(i-n) & \rm{for }n+1\leq i\leq n+m.
\end{cases}
\]
A \defi{homogeneous invariant right ideal} $\II=\bigoplus_{n\geq 0}\II_n$ is a homogeneous right ideal in $\TT$ with the property that each homogeneous component $\II_n$ is a left ideal in the group algebra $\K[\S_n]$. If $A\subset\TT$, we write $\II(A)$ for the smallest homogeneous invariant right ideal containing $A$. If $S,S'$ are Young tableaux with set of entries $[k]=\{1,2,\cdots,k\}$, we say that $S'$ \defi{dominates} $S$ if for every $i\in[k]$, writing $j$ (resp. $j'$) for the column of $S$ (resp. $S'$) containing $i$, then $j\geq j'$. We have the following (see Theorem \ref{thm:idealtab} for a stronger statement, and Theorem \ref{thm:idealdtab} for a partially symmetrized version):

\begin{theorem}\label{thm:idgentensoralg}
 Let $k\leq n$ be positive integers, let $\ll\vdash n$, $\mu\vdash k$ be partitions with $\mu\subset\ll$, and let $T$ be a Young tableau of shape $\ll$ containing a Young subtableau $S$ of shape $\mu$ and set of entries $[k]$. Let $\mc{S}$ be the set of Young tableaux $S'$ that have shape $\delta\vdash k$, $\delta\subset\ll$, set of entries $[k]$, and dominate $S$. We have
\[\mfc_{\ll}(T)\in\II(\mfc_{\delta}(S'):S'\in\mc{S}).\]
\end{theorem}

\begin{example}\label{ex:idgentensoralg}
 Take $n=7$, $k=5$, $\ll=(4,2,1)$, $\mu=(3,2)$ and
\[\Yvcentermath1
T=\young(1236,45,7),\quad S=\young(123,45).
\]
The set $\mc{S}$ in Theorem \ref{thm:idgentensoralg} consists of the Young tableaux
\[\Yvcentermath1\young(123,45),\quad\young(123,4,5),\quad\young(153,4,2),\quad\young(12,45,3),\quad\young(12,43,5),\quad\young(13,45,2),
\]
together with the ones obtained from them by permuting the entries within each column.
\end{example}

The structure of the paper is as follows. In Section~\ref{sec:prelim} we give some preliminary definitions and results on Young tableaux and Young symmetrizers. In Section~\ref{sec:gentensor} we describe the generic tensor algebra and a symmetric version of it, together with some consequences of Theorem~\ref{thm:main} to the ideal structure of these generic algebras. In Section~\ref{sec:proofThm2} we prove Theorem~\ref{thm:yngconsec}, and in Section~\ref{sec:proofThm1} we use an inductive argument based on Theorem~\ref{thm:yngconsec} in order to prove Theorem~\ref{thm:main}.

\section{Preliminaries}\label{sec:prelim}

Given a finite set $A$ of size $n=|A|$, we write $\S_A$ for the symmetric group of permutations of $A$. When $A=\{1,\cdots,n\}$ we simply denote $\S_A$ by $\S_n$. If $B\subset A$ then we regard $\S_B$ as a subgroup of $\S_A$ in the natural way. We fix a field $\K$ of characteristic zero, and write $\K[\S_A]$ for the group algebra of $\S_A$.

When $\ll\vdash n$ is a partition of $n$, $D_{\ll}=\{(i,j):1\leq j\leq \ll_i\}$ is the associated \defi{Young diagram}. A \defi{Young tableau $T$} of shape $\ll$ and set of entries $A$ is a bijection $T:D_{\ll}\to A$. We represent Young diagrams (resp. Young tableaux) pictorially as collections of left-justified rows of boxes (resp. filled boxes) with $\ll_i$ boxes in the $i$-th row, as illustrated in the following example: for $A=\{a,b,c,d,e,f,g\}$ and $\ll=(4,2,1)$, we take
\[\Yvcentermath1 D_{\ll}=\yng(4,2,1),\quad T=\young(cabg,ed,f).\]
The $(2,1)$-entry of $T$ is $e$. The $i$-th row of $T$ is the set $R_i(T)=\{T(i,j):1\leq j\leq\ll_i\}$, and its $j$-th column $C_j(T)$ is defined analogously. In the example above $C_2(T)=\{a,d\}$. The \defi{conjugate} of a partition $\ll$ (resp. Young diagram $D_{\ll}$/tableau $T$) is obtained by reversing the roles of rows and columns, and is denoted $\ll'$ (resp. $D_{\ll'},T'$). In our example $T'$ has shape $\ll'=(3,2,1,1)$. If $\mu$ is another partition, we write $\mu\subset\ll$ if $\mu_i\leq\ll_i$ for all $i$. If $\mu\subset\ll$ then $D_{\mu}\subset D_{\ll}$ and we call the restriction $S=T|_{D_{\mu}}$ the \defi{Young subtableau} of $T$ of shape $\mu$. For example
\[\Yvcentermath1 S=\young(cab,e)\]
is a Young subtableau of shape $\mu=(3,1)$ of the above $T$.

If $X$ is any subset of $A$, we write
\[\mfa(X)=\sum_{\s\in\S_X} \s,\quad \mfb(X)=\sum_{\t\in\S_X} \rm{sgn}(\t)\cdot \t.\]
For $\delta\in\S_A$ we have $\mfa(\delta(X))=\delta\cdot \mfa(X)\cdot\delta^{-1}$, and similarly for $\mfb(\delta(X))$. If $\delta\in\S_X$ then $\delta\cdot \mfa(X)=\mfa(X)$ and $\delta\cdot \mfb(X)=\rm{sgn}(\delta)\cdot \mfb(X)$, so $\mfa(X)^2=|X|!\cdot \mfa(X)$ and $\mfb(X)^2=|X|!\cdot \mfb(X)$. If $a\in A\setminus X$ and if we let $z=\sum_{x\in X}(a,x)$, where $(a,x)\in\S_A$ denotes the transposition of $a$ with $x$, then
\begin{equation}\label{eq:symmskewsymmadda}
\begin{aligned}
\mfa(X\cup\{a\})&=\mfa(X)\cdot(1+z)=(1+z)\cdot \mfa(X),\\
\mfb(X\cup\{a\})&=\mfb(X)\cdot(1-z)=(1-z)\cdot \mfb(X).
\end{aligned}
\end{equation}
If $X,Y\subset A$ with $|X\cap Y|\geq 2$ then $\mfa(X)\cdot \mfb(Y)=0$: to see this, consider a transposition $\t\in\S_{X\cap Y}$ and note that
\[\mfa(X)\cdot \mfb(Y)=(\mfa(X)\cdot\t)\cdot \mfb(Y)=\mfa(X)\cdot(\t\cdot \mfb(Y))=\mfa(X)\cdot(-\mfb(Y)).\]
A similar argument shows that $\mfb(Y)\cdot\mfa(X)=0$.

We define the \defi{row} and \defi{column subgroups} associated to $T$ as the subgroups of $\S_A$
\begin{equation}\label{eq:rowcolsubgps}
\mc{R}_T=\prod_i \S_{R_i(T)}\rm{ and }\mc{C}_T=\prod_j \S_{C_j(T)}.
\end{equation}
The \defi{Young symmetrizer} $\mfc_{\ll}(T)$ is defined by
\begin{equation}\label{eq:Youngsymm}
\begin{aligned}
\mfc_{\ll}(T)&=\mfa_{\ll}(T)\cdot \mfb_{\ll}(T),\rm{ where} \\
\mfa_{\ll}(T)=\sum_{\s\in\mc{R}_T} \s=\prod_{i=1}^{\ll'_1} \mfa(R_i(T)),&\quad \mfb_{\ll}(T)=\sum_{\t\in\mc{C}_T} \rm{sgn}(\t)\cdot \t=\prod_{j=1}^{\ll_1} \mfb(C_j(T)).
\end{aligned}
\end{equation}
$\S_A$ acts naturally on the set of Young tableaux with set of entries $A$, and we have $\mfc_{\ll}(\delta\cdot T)=\delta\cdot \mfc_{\ll}\cdot\delta^{-1}$, with similar formulas for $\mfa_{\ll}(\delta\cdot T)$ and $\mfb_{\ll}(\delta\cdot T)$.

It follows from (\ref{eq:symmskewsymmadda}) and the fact that $\mfa(X)\cdot\mfb(Y)=0$ when $|X\cap Y|\geq 2$ that if $i\neq j$ are such that $\ll'_i\leq\ll'_j$ and $a\in C_i(T)$ then
\begin{equation}\label{eq:Garnircol}
 \mfc_{\ll}(T)\cdot\left(1-\sum_{x\in C_j(T)}(a,x)\right)=0.
\end{equation}
This relation is an instance of the \defi{Garnir relations} \cite[Section~7]{james}.

The \defi{hook} of $\ll$ centered at $(x,y)$ is the subset $H_{x,y}=\{(x,j):j\geq y\}\cup\{(i,y):i\geq x\}\subset D_{\ll}$. Its \defi{length} is the size of $H_{x,y}$. We write 
\begin{equation}\label{eq:prodhooks}
\a_{\ll}=\prod_{(i,j)\in D_{\ll}}|H_{i,j}|.
\end{equation}
It follows from the Hook Length Formula \cite[Section~4.1]{ful-har} and \cite[Lemma~4.26]{ful-har} that
\begin{equation}\label{eq:idempyngsymm}
\mfc_{\ll}(T)^2=\a_{\ll}\cdot \mfc_{\ll}(T).
\end{equation}

\section{Ideals in the generic tensor algebra}\label{sec:gentensor}

In this section we illustrate some applications of Theorem~\ref{thm:main} to the structure of the ideals in the generic tensor algebra and its partial symmetrizations. Special instances of Theorems~\ref{thm:idealtab} and~\ref{thm:idealdtab} below were used in \cites{rai-GSS,oed-rai} in the study of the equations and homogeneous coordinate rings of the secant line and tangential varieties of Segre--Veronese varieties. We illustrate the relevant constructions in the case of the Veronese variety, the extension to the multigraded situation being just a matter of notation.

We write $Vec_{\K}$ for the category of finite dimensional vector spaces, and $Set$ for the category of finite sets, where morphisms are bijective functions. Note that $\rm{Hom}_{Set}(A,A)=\S_A$, so for any functor $\FF:Set\to Vec_{\K}$, $\FF(A)=\FF_A$ is a $\S_A$--representation. By an \defi{element} of $\FF$ we mean an element of $\FF_A$ for some $A\in Set$. Consider $\TT:Set\to Vec_{\K}$ the functor which assigns to a set $A$ with $|A|=n$, the vector space $\TT_A$ having a basis consisting of symbols $z_{\a}$ where $\a$ runs over the set of bijections between $[n]=\{1,\cdots,n\}$ and $A$. We think of $z_{\a}$ as the tensor $z_{\a(1)}\oo z_{\a(2)}\oo\cdots\oo z_{\a(n)}$. We have a natural multiplication on $\TT$, namely for $A,B\in Set$ we have a map $\mu_{A,B}:\TT_A\oo\TT_B\to\TT_{A\sqcup B}$ which extends linearly the concatenation of tensors. In terms of the symbols $z_{\a}$, if $\a:[r]\to A$ and $\b:[s]\to B$ are bijections, then $\mu_{A,B}(z_{\a},z_{\b})=z_{\c}$, where $\c(i)=\a(i)$ for $i=1,\cdots,
r$, and 
$\c(i)=\b(i-r)$ for $i=r+1,\cdots,r+s$. We will often write $\mu_{A,B}(x\oo y)$ simply as $x\cdot y$. We call $\TT$ the \defi{generic tensor algebra}. A \defi{(right) ideal} in $\TT$ is a subfunctor $\II\subset\TT$ with the property that $\II_A\cdot\TT_B\subset\II_{A\sqcup B}$ for all $A,B\in Set$. The ideal $\II(\mc{E})$ generated by a set $\mc{E}$ of elements of $\TT$ is the smallest ideal that contains them.

Consider a partition $\ll\vdash n$ and a Young tableau $F:D_{\ll}\to [n]$. We define the \defi{Young tabloid} associated to $F$ to be the collection $[F]=\{t_F(A,T)\in\TT_A\}$ of elements of $\TT$ obtained as follows. For any $A\in Set$ with $|A|=n$, and any Young tableau $T:D_{\ll}\to A$ we let $t_F(A,T)=\mfc_{\ll}(T)\cdot z_{T\circ F^{-1}}$. We represent a Young tabloid just as a Young tableaux, but with the horizontal lines removed:
\[
\ytableausetup{boxsize=normal,tabloids,aligntableaux=center}
[F]=\ytableaushort{
1236, 45, 7
}
\]
Note that we only construct tabloids from tableaux with entries in $[n]$. If $F_i$ are Young tableaux of shape $\ll$ then a \defi{relation} $\sum_{i}a_i\cdot[F_i]=0$ means that for any choice of $A\in Set$ with $|A|=n$ and of a Young tableau $T:D_{\ll}\to A$, we have $\sum_i a_i\cdot t_{F_i}(A,T)=0\in\TT_A$. Young tabloids are skew-symmetric in columns, and satisfy the Garnir relations \cite[Section~7]{james}, also known as shuffling relations \cite[Section~2.1]{weyman}:

\begin{lemma}\label{lem:shuffling}
 Let $F:D_{\ll}\to [n]$ be a Young tableaux. The following relations hold
 
(a) If $\s\in\mc{C}_F$ (where $\mc{C}_F$ is as in (\ref{eq:rowcolsubgps})) then $[\s\circ F]=\rm{sgn}(\s)\cdot [F]$.

(b) If $X\subset C_i(F)$ and $Y\subset C_{i+1}(F)$ with $|X\cup Y|>|C_i(F)|$ then
\[\sum_{\s\in\S_{X\cup Y}}\rm{sgn}(\s)\cdot[\s\circ F]=0.\]
\end{lemma}

We define the \defi{ideal generated by $F$} by $\II(F)=\II([F])$, and more generally we can define the ideal generated by a collection of Young tableaux. Observe that in fact $\II(F)=\II(t_F(A,T))$ for any $A\in Set$ with $|A|=n$ and any Young tableau $T$ of shape $\ll$ and entries in $A$. To see this, consider another pair $(A',T')$ and the corresponding element $t_F(A',T')\in\TT_{A'}$. Let $\phi=T'\circ T^{-1}\in\rm{Hom}_{Set}(A,A')$. We have
\[t_F(A',T')=\mfc_{\ll}(T')\cdot z_{T'\circ F^{-1}}=\phi(\mfc_{\ll}(T)\cdot z_{T\circ F^{-1}})=\phi(t_F(A,T)),\]
so in fact any subfunctor of $\TT$ that contains $t_F(A,T)$ will also contain $t_F(A',T')$, and vice versa.

We write $(x_1,y_1)\prec(x_2,y_2)$ if $y_1<y_2$ and $(x_1,y_1)\preceq(x_2,y_2)$ if $y_1\leq y_2$. If $(x_i,y_i)$ are the coordinates of a box $b_i$ in a Young tableau $T$, then $(x_1,y_1)\prec(x_2,y_2)$ means that $b_1$ is contained in a column of $T$ situated to the left of the column of $b_2$.

\begin{theorem}\label{thm:idealtab}
 Let $k\leq n$ be positive integers, let $\ll\vdash n$, $\mu\vdash k$ be partitions with $\mu\subset\ll$, and let $F:D_{\ll}\to[n]$ be a Young tableau with $F^{-1}([k])=D_{\mu}$. Denote by $\mc{G}$ the collection of Young tableaux $G:D_{\ll}\to [n]$ with the properties
\begin{enumerate}
 \item $G^{-1}(i)\preceq F^{-1}(i)$ for all $i\in [k]$, with $G^{-1}(i)\prec F^{-1}(i)$ for at least one $i\in [k]$.
 \item $G^{-1}([k])=D_{\delta}$ for some partition $\delta\subset\ll$.
\end{enumerate}
Writing $F_0=F|_{D_{\mu}}$ we have
 \begin{equation}\label{eq:IF0G}
  [F]\in\II([F_0])+\II([G]:G\in\mc{G}).
 \end{equation}
In particular, if we write $G_0$ for the restriction of $G$ to $G^{-1}([k])$ then
 \begin{equation}\label{eq:IF0Gbar}
 [F]\in\II([F_0])+\II([G_0]:G\in\mc{G}).
 \end{equation}
\end{theorem}

\begin{example}\label{ex:idealtab}
 Take $n=7$, $k=5$, $\ll=(4,2,1)$, $\mu=(3,2)$ and
\[
\ytableausetup{boxsize=normal,tabloids,aligntableaux=center}
[F]=\ytableaushort{
1236, 45, 7
}
\quad\textrm{so that}\quad
[F_0]=\ytableaushort{
123, 45
}
\]
Modulo $\II([F_0])$, $[F]$ is a linear combination of tabloids of shape $\ll$ containing one of
\[
\ytableausetup{boxsize=normal,tabloids,aligntableaux=center}
\ytableaushort{
123, 4, 5
}
\quad,\quad
\ytableaushort{
153, 4, 2
}
\quad,\quad
\ytableaushort{
12, 45, 3
}
\quad,\quad
\ytableaushort{
12, 43, 5
}
\quad,\quad
\ytableaushort{
13, 45, 2
}
\quad.
\]
\end{example}

\begin{proof}[Proof of Theorem \ref{thm:idealtab}] Consider any $A\in Set$, $|A|=n$ and a Young tableau $T:D_{\ll}\to A$. Let $T_0=T|_{D_{\mu}}$ and $A_0=T(D_{\mu})$. Applying Theorem \ref{thm:main} with $S=T_0$ we have
\begin{equation}\label{eq:Thm1}
\mfc_{\ll}(T)\cdot \mfc_{\mu}(T_0)=\a_{\mu}\cdot \mfc_{\ll}(T)+ \mfc_{\ll}(T)\cdot\left(\sum_{\s\in L(T;T_0),\s\neq\1}m_{\s}\cdot\s\right). 
\end{equation}
Consider the bijection $\c:[n-k]\to (A\setminus A_0)$ defined by $\c(j)=(T\circ F^{-1})(j+k)$, and observe that $z_{T\circ F^{-1}}=z_{T_0\circ F_0^{-1}}\cdot z_{\c}$. Multiplying both sides of (\ref{eq:Thm1}) by $z_{T\circ F^{-1}}$ yields
\[\II([F_0])\ni\mfc_{\ll}(T)\cdot t_{F_0}(A_0,T_0)\cdot z_{\c}=\a_{\mu}\cdot t_F(A,T) + \sum_{\s\in L(T;T_0),\s\neq\1}m_{\s}\cdot\mfc_{\ll}(T)\cdot\s\cdot z_{T\circ F^{-1}}.\]
To prove (\ref{eq:IF0G}) it is thus enough to show that for $\s\in L(T;T_0)$, $\s\neq\1$, $\mfc_{\ll}(T)\cdot\s\cdot z_{T\circ F^{-1}}$ is a linear combination of $t_{G_i}(A,T)$ for $G_i\in\mc{G}$. Consider the Young tableaux $G=F\circ T^{-1}\circ\s^{-1}\circ T$. Since $\s\cdot z_{T\circ F^{-1}}=z_{\s\circ T\circ F^{-1}}=z_{T\circ G^{-1}}$, we have $\mfc_{\ll}(T)\cdot\s\cdot z_{T\circ F^{-1}}=t_G(A,T)$.

We prove that $G$ satisfies condition (1) of the theorem. To see this, write $(x,y)=F^{-1}(i)$ and let $b=T(x,y)$. We have $G^{-1}(i)=T^{-1}(\s(b))$, and by the definition of $L(T;T_0)$, either $\s(b)=b$, or $\s(b)$ lies strictly to the left of $b$ in the Young tableau $T$. It follows that 
\[G^{-1}(i)=T^{-1}(\s(b))\preceq T^{-1}(b)=(x,y)=F^{-1}(i),\]
with equality if and only if $b=\s(b)$. Since $\1$ is the only permutation in $L(T;T_0)$ that fixes all $b\in A_0=(T\circ F^{-1})([k])$, the conclusion follows. 

In general it won't be the case that $G$ also satisfies (2), but we can perform a straightening algorithm based on Lemma \ref{lem:shuffling} to write $[G]$ as a linear combination of $[G_i]$ with $G_i\in\mc{G}$. First of all, using part (a) of Lemma \ref{lem:shuffling}, we may assume that for each column $C_i(G)$ the entries in $[k]$ appear before those in $[n-k]$. If $G^{-1}([k])$ is not the Young diagram of a partition, it means that we can find consecutive columns $C_i(G)$, $C_{i+1}(G)$ such that $|C_i(G)\cap [k]|<|C_{i+1}(G)\cap [k]|$. Let $X=C_i(G)\setminus[k]$ and $Y=C_{i+1}(G)\cap[k]$. Applying part (b) of Lemma \ref{lem:shuffling}, we can write $[G]$ as a linear combination of $[G_j]$ where $C_i(G)\cap[k]\subsetneq C_i(G_j)\cap[k]$. Condition (1) will be satisfied by the $G_j$'s since the shuffling relation only moves elements of $[k]$ to the left, so iterating the process yields the desired conclusion.

To prove (\ref{eq:IF0Gbar}), apply (\ref{eq:IF0G}) and induction to each $G\in\mc{G}$ to conclude that $\II([G]:G\in\mc{G})\subset\II([G_0]:G\in\mc{G})$.
\end{proof}

\subsection{Partially symmetric generic tensor algebras}\label{subsec:symmetric}

Starting from the generic tensor algebra $\TT$, one can construct by partial symmetrization other generic algebras that come up naturally in the study of varieties of tensors with a GL-action. We describe the generic algebras relevant to the study of the secant and tangential variety of a Veronese variety, and leave it to the interested reader to perform the construction in other cases of interest.

Write $Set^d$ for the subcategory of $Set$ consisting of sets whose size is divisible by $d$. Consider the functor $\SS^{(d)}:Set^d\to Vec_{\K}$ which assigns to a set $A\in Set^d$, $|A|=nd$, the vector space $\SS^{(d)}_A$ with basis consisting of monomials $z_{\A}=z_{A_1}\cdots z_{A_n}$ in commuting variables $z_{A_i}$, where $\A$ runs over partitions $A=A_1\sqcup\cdots\sqcup A_n$ with $|A_i|=d$. In the work of Sam and Snowden \cite{sam-snowden}, $\SS^{(d)}$ is the \defi{twisted commutative algebra} $\Sym(\Sym^d(\bb{C}^{\infty}))$. If we let $\TT^{(d)}=\TT|_{Set^d}$ be the restriction of $\TT$ to $Set^d$, then there is a natural map $\pi:\TT^{(d)}\to\SS^{(d)}$ which is $\K$-linear and is defined on the elements of the form $z_{\a}$ as follows. For every $A\in Set^d$, $|A|=nd$, and bijection $\a:[nd]\to A$, we consider the partition $\A$ of $A$ obtained by letting $A_i=\a(\{d(i-1)+1,\cdots,di\})$, and define $\pi(z_{\a})=z_{\A}$. $\pi$ is surjective and multiplicative, so any ideal in $\SS^{(d)}$ is the 
image of an ideal in $\TT^{(d)}$. It follows that Theorem~\ref{thm:idealtab} can be used to derive an analogous result for ideals in $\SS^{(d)}$, which we explain next.

For a partition $\ll\vdash nd$, we define a \defi{Young$^d_n$ tableau} of shape $\ll$ to be a function $F:D_{\ll}\to[n]$ with the property that $|F^{-1}(i)|=d$ for all $i\in[n]$. If $T:D_{\ll}\to A$ is a Young tableau then we write $T\circ F^{-1}$ for the partition $\A$ of $A$ with $A_i=T(F^{-1}(i))$. We define the \defi{Young$^d_n$ tabloid} associated to $F$ as before, $[F]=\{t_F(A,T)=\mfc_{\ll}(T)\cdot z_{T\circ F^{-1}}\in\SS^{(d)}_A\}$. Note that replacing $F$ by $\s\circ F$ for $\s\in\S_n$ permutes the parts of the partition $\A$, but preserves $z_{\A}$ because of the commutativity of the $z_{A_i}$'s. It follows that $[F]=[\s\circ F]$ for all $\s\in\S_n$. We represent Young$^d_n$ tabloids just as the Young tabloids, allowing each entry of $[n]$ to occur exactly $d$ times: for $d=3$, $n=4$, $\ll=(6,3,2,1)$, a typical Young$^3_4$ tabloid of shape $\ll$ would be
\[
\ytableausetup{boxsize=normal,tabloids,aligntableaux=center}
[F] = \ytableaushort{
123133, 244, 12, 4
}\ =\ \ytableaushort{
134144, 322, 13, 2
} 
\]
Note that part (a) of Lemma \ref{lem:shuffling} implies that if $F$ has repeated entries in some column, then $[F]=0$. This is the case in the example above.

If $X=\{(x_i,y_i):i\in[d]\}$, $X'=\{(x'_i,y'_i):i\in[d]\}$, with $y_1\leq y_2\leq\cdots$, $y'_1\leq y'_2\leq\cdots$, we write $X\preceq X'$ if $y_i\leq y_i'$ for all $i\in[d]$, and $X\prec X'$ if $X\preceq X'$ and $y_i<y'_i$ for some $i$. With this notation, Theorem \ref{thm:idealtab} translates almost without change in the partially symmetric setting:

\begin{theorem}\label{thm:idealdtab}
 Let $k\leq n$ and $d$ be positive integers, let $\ll\vdash nd$, $\mu\vdash kd$ be partitions with $\mu\subset\ll$, and let $F:D_{\ll}\to[n]$ be a Young$^d_n$ tableau with $F^{-1}([k])=D_{\mu}$. Denote by $\mc{G}$ the collection of Young$^d_n$ tableaux $G:D_{\ll}\to [n]$ with the properties
\begin{enumerate}
 \item $G^{-1}(i)\preceq F^{-1}(i)$ for all $i\in [k]$, with $G^{-1}(i)\prec F^{-1}(i)$ for at least one $i\in [k]$.
 \item $G^{-1}([k])=D_{\delta}$ for some partition $\delta\subset\ll$.
\end{enumerate}
Writing $F_0=F|_{D_{\mu}}$ we have
 \[[F]\in\II([F_0])+\II([G]:G\in\mc{G}).\]
In particular, if we write $G_0$ for the restriction of $G$ to $G^{-1}([k])$ then
 \[[F]\in\II([F_0])+\II([G_0]:G\in\mc{G}).\]
\end{theorem}
The condition $G\in\mc{G}$ can be restated simply by saying that when going from $F$ to $G$, each entry of $F$ contained in $D_{\mu}$ either remains in the same column, or is moved to the left, the latter situation occurring for at least one such entry. The exact location of an entry within a column is irrelevant due to part (a) of Lemma \ref{lem:shuffling}.

\begin{proof} Consider a Young tableau $\tl{F}:D_{\ll}\to[nd]$ which is a lifting of the Young$^d_n$ tableau $F$, i.e. $\tl{F}$ induces a bijection between $F^{-1}(i)$ and $\{d(i-1)+1,\cdots,di\}$ for all $i\in[n]$. By construction we have that $\pi([\tl{F}])=[F]$, i.e. $\pi(t_{\tl{F}}(A,T))=t_F(A,T)$ for all $A$ with $|A|=nd$ and all Young tableaux $T:D_{\ll}\to A$. Letting $\tl{F}_0=\tl{F}|_{D_{\mu}}$, we have that $\pi([\tl{F}_0])=[F_0]$. If $\tl{G}$ is any Young tableau satisfying conditions (1) and (2) of Theorem~\ref{thm:idealtab} (with $k$ and $n$ replaced by $kd$ and $nd$ respectively) then the Young$^d_n$ tableau $G:D_{\ll}\to[n]$ obtained by $G=\pi'\circ\tl{G}$, where $\pi'(j)=i$ when $j\in\{d(i-1)+1,\cdots,di\}$, satisfies conditions (1) and (2) of Theorem~\ref{thm:idealdtab}, and moreover $\pi([\tl{G}])=[G]$. The conclusion of the theorem then follows from that of Theorem~\ref{thm:idealtab}. 
\end{proof}

We end with a series of examples of ideals in the generic algebra $\SS^{(d)}$, explaining their relevance to the study of spaces of tensors. Before that, we introduce one last piece of notation.

For a functor $\FF:Set^d\to Vec_{\K}$ and a partition $\ll\vdash nd$, write $\FF_{\ll}$ for the subfunctor that assigns to a set $A\in Set^d$ the $\ll$--isotypic component $(\FF_A)_{\ll}$ of the $\S_A$--representation $\FF_A$ (note that $(\FF_{\ll})_A=(\FF_A)_{\ll}$). A choice of a set $A$ with $|A|=nd$ and of a Young tableau $T:D_{\ll}\to A$, gives rise to a vector space $c_{\ll}(T)\cdot (\FF_A)_{\ll}$ of dimension equal to the multiplicity $m_{\ll}(\FF)$ of the irreducible $\S_A$--representation $[\ll]$ inside $\FF_A$. We call this space a \defi{$\ll$--highest weight space} of $\FF$ and denote it by $\rm{hwt}_{\ll}(\FF)$. We call the elements of $\rm{hwt}_{\ll}(\FF)$ \defi{$\ll$--covariants} of $\FF$. Note that there are choices in the construction of the $\ll$--covariants of $\FF$, but the subfunctor of $\FF$ that they generate is $\FF_{\ll}$, which is independent of these choices.

Taking $\FF=\SS^{(d)}$ and $\ll\vdash nd$, the $\ll$--covariants of $\SS^{(d)}$ are just linear combinations of Young$^d_n$ tabloids of shape $\ll$. We have that $m_{\ll}(\SS^{(d)})$ coincides with the multiplicity of the Schur functor $S_{\ll}$ inside the plethysm $\Sym^n\circ\Sym^d$. We call $\SS^{(d)}$ the \defi{generic version} of the polynomial ring $S=\Sym(\Sym^d V)$, which is the homogeneous coordinate ring of the projective space $\bb{P}(\Sym^d V)$. More generally, we write $Sch_{\K}$ for the category of $\K$--schemes, and consider a contravariant functor $X:Vec_{\K}\to Sch_{\K}$, with the property that $X(V)\subset \bb{P}(\Sym^d V)$ is a closed subscheme. The ideal of equations $I(X(V))$ and homogeneous coordinate ring $\K[X(V)]$ define polynomial functors $I_X,S_X:Vec_{\K}\to Vec_{\K}$. They have corresponding \defi{generic versions} $\II_X\subset\SS^{(d)}$ and $\SS^{(d)}_X=\SS^{(d)}/\II_X$ defined as follows. For $A$ with $|A|=nd$, we consider a vector space $V_A$ with a basis indexed by the 
elements of $A$. The choice of basis on $V_A$ gives rise to a maximal torus $T_A\subset\rm{GL}(V_A)$ of diagonal matrices, and there is a natural identification between the $(1,1,\cdots,1)$--weight space of $\Sym^n(\Sym^d V_A)$ and the vector space $\SS^{(d)}(A)$. Via this identification, we let $\II_X(A)$ be the subspace of $\SS^{(d)}(A)$ that corresponds to the $(1,1,\cdots,1)$--weight space of $I_X(V)\subset\Sym^n(\Sym^d V)$. The information encoded by $I_X,S_X$ is equivalent to that of $\II_X,\SS^{(d)}_X$ (see also the discussion on polarization and specialization from \cite[Section~3C]{rai-GSS}, and \cite{sam-snowden}). 

\begin{example}[Generic ideals of subspace varieties]\label{ex:subspace}
 Denote by $\II^{<k}\subset\SS^{(d)}$ the ideal generated by the $\ll$--covariants of $\SS^{(d)}$, where $\ll$ runs over partitions with at least $k$ parts. $\II^{<k}$ is the generic version of the ideal of a subspace variety \cite[Section~7.1]{landsberg}: we have $\II^{<k}=\II_X$, where $X:Vec_{\K}\to Sch_{\K}$ is defined by letting $X(V)$ be the union of all the subspaces $\bb{P}(\Sym^d W)\subset\bb{P}(\Sym^d V)$, where $W$ runs over the $(k-1)$--dimensional quotients of $V$. When $k=2$, $X=Ver_d$ is the functor which associates to $V$ the \defi{$d$--th Veronese embedding} of $\bb{P}V$. 
\end{example}

\begin{example}[Covariants associated to graphs]\label{ex:graphcovariant}
 Given a graph $Q$, we write $V(Q)$ for the set of vertices, and $E(Q)$ for the multiset of edges (we allow multiple edges between two vertices). If $Q$ is an unlabeled graph with $n$ vertices and $e$ edges, with the property that to any vertex there are at most $d$ incident edges, then one constructs a $\ll$--covariant $p(Q)\in\SS$ for $\ll=(nd-e,e)$, as follows. Choose a labeling of the vertices of $Q$ with elements of $[n]$ and consider a Young$^d_n$ tableau $F:D_{\ll}\to [n]$ having a column of size two with entries $x,y$ (in some order) for each edge $xy\in E(Q)$. The columns of size one of $F$ are such that each element of $[n]$ appears exactly $d$ times in $F$. As before, $[F]$ denotes the associated Young$^d_n$ tabloid. For example, when $d=3$, $r=4$ and $e=5$, typical $Q$ and $[F]$ look like
\[\ytableausetup{boxsize=normal,tabloids,aligntableaux=center}
\Yvcentermath1
 \xymatrix@=15pt{
 & & \ccircle{3} \ar@{-}[dd] \ar@{-}[dr] & &\\
Q\quad = \quad & \ccircle{2} \ar@{-}[ur] \ar@{-}@<-.5ex>[dr] \ar@{-}@<.5ex>[dr] & & \ccircle{4} & \quad\leadsto\quad [F]\quad=\quad\ytableaushort{1112344,22334} \\ 
 & & \ccircle{1} & &\\
}
\]
\noindent Write $[Q]$ for the tabloid $[F]$. There are choices in the construction of $[Q]$, but the ideal $\II([Q])$ it generates inside $\SS^{(d)}$ is independent of these choices and is denoted $\II(Q)$. More generally, if $\mc{Q}$ is a family of graphs, $\II(\mc{Q})=\II([Q]:Q\in\mc{Q})$ is the ideal generated by the corresponding tabloids.
\end{example}

A direct consequence of Theorem \ref{thm:idealdtab} is the following

\begin{proposition}\label{prop:idealgraphs}
 Let $Q'$ be a subgraph of a graph $Q$. We have $\II(Q)\subset\II(\mc{Q})$ where $\mc{Q}$ is the set of graphs $\tilde{Q}$ with $V(Q')=V(\tilde{Q})$, $E(Q')\subseteq E(\tilde{Q})$ and $|E(\tilde{Q})|\leq |E(Q)|$.
\end{proposition}

\begin{example}[Generic ideals of secant line varieties \cite{rai-GSS}]\label{ex:secant}
 Consider $\s_2:Vec_{\K}\to Sch_{\K}$, defined by letting $\s_2(V)$ be the variety of secant lines to $Ver_d(V)$. We have $\II(\s_2)=\II^{<3}+\II^{\Delta}$, where $\II^{\Delta}$ is the ideal generated by graphs containing a \defi{triangle} (i.e. a complete subgraph on $3$ vertices). It follows from Proposition \ref{prop:idealgraphs} that $\II^{\Delta}$ is generated in degree three by the graphs on $3$ vertices that contain a triangle.
\end{example}

\begin{example}[Generic ideals of tangential varieties \cite{oed-rai}]\label{ex:tangential}
 Consider $\t:Vec_{\K}\to Sch_{\K}$, defined by letting $\t(V)$ be the tangential variety to $Ver_d(V)$. We have $\II(\tau)=\II^{<3}+\II^{rich}$, where $\II^{rich}$ is the ideal generated by \defi{rich} graphs, i.e. graphs with more edges than vertices. $\II^{rich}$ is generated by graphs with at most $4$ vertices. In particular, the ideal of $\tau$ is generated in degree at most $4$.
\end{example}

\section{Proof of Theorem \ref{thm:yngconsec}}\label{sec:proofThm2}

We fix a Young tableau $S$, and the Young tableau $T$ obtained from $S$ by adding one box with entry $a$ in position $(u,v)$. We write $\mu$ (resp. $\ll$) for the shape of $S$ (resp. $T$). We begin by rewriting equation (\ref{eq:prodconsec}) in a more convenient form. We define for $1\leq j\leq\mu_1$
\begin{equation}\label{eq:colsum}
z_j=\sum_{b\in C_j(S)}(a,b)\in\K[\S_n].
\end{equation}
Writing $\mu!=\prod_i \mu_i!$, we note that
 \begin{equation}\label{eq:absymmetrizers}
 \begin{aligned}
  \mfa_{\ll}(T)=&\frac{\mfa_{\ll}(T)}{\mu!}\cdot \mfa_{\mu}(S),\quad \mfb_{\ll}(T)\overset{(\ref{eq:symmskewsymmadda})}{=}\mfb_{\mu}(S)\cdot(1-z_v),\\
  &\mfc_{\ll}(T)=\frac{\mfa_{\ll}(T)}{\mu!}\cdot \mfc_{\mu}(S)\cdot(1-z_v).
 \end{aligned}
 \end{equation}
Using (\ref{eq:absymmetrizers}), equation (\ref{eq:prodconsec}) becomes after multiplying by $\mu!$
\begin{equation}\label{eq:betterprodconsec}
\mfa_{\ll}(T)\cdot \mfc_{\mu}(S)\cdot(1-z_v)\cdot \mfc_{\mu}(S)=\mfa_{\ll}(T)\cdot \mfc_{\mu}(S)\cdot\a_{\mu}\cdot(1-z_v)\cdot\prod_{i=1}^{\tl{m}}\left(1-\frac{\tl{x}_i}{\tl{r}_i}\right).  
\end{equation}
Consider the blocks $B_1,\cdots,B_m$ of the Young tableau obtained by removing the first $(v-1)$ columns of $S$, denote by $l_i$ (resp. $h_i$) their lengths (resp. heights), let 
\begin{equation}\label{eq:xi}
x_i=\sum_{b\in B_i}(a,b) 
\end{equation}
and let $r_i$ denote the length of the hook of $\mu$ centered at $(h_i,v)$,
\begin{equation}\label{eq:hookri}
 r_i=l_1+\cdots+l_i+h_1-h_i.
\end{equation}
Note that we have two possibilities:
\begin{enumerate}
 \item $\mu'_v>\mu'_{v+1}$: in which case $m=\tl{m}+1$, $B_1=C_v(S)$, $B_i=\tl{B}_{i-1}$ for $i>1$.
 \item $\mu'_v=\mu'_{v+1}$: in which case $m=\tl{m}$, $B_1=C_v(S)\cup\tl{B}_1$, $B_i=\tl{B}_i$ for $i>1$.
\end{enumerate}

\begin{lemma}\label{lem:betterprodconsec}
 With the notation above, we have
 \begin{equation}\label{eq:czv}
  \mfc_{\mu}(S)\cdot(1-z_v)\cdot\prod_{i=1}^{\tl{m}}\left(1-\frac{\tl{x}_i}{\tl{r}_i}\right)=\mfc_{\mu}(S)\cdot\prod_{i=1}^{m}\left(1-\frac{x_i}{r_i}\right).
 \end{equation}
 \begin{equation}\label{eq:czvc}
  \mfc_{\mu}(S)\cdot(1-z_v)\cdot \mfc_{\mu}(S)=\mfc_{\mu}(S)\cdot\left(1-\frac{x_1}{r_1}\right)\cdot \mfc_{\mu}(S).
 \end{equation}
In particular, equation (\ref{eq:betterprodconsec}) is equivalent to
\begin{equation}\label{eq:consec}
\mfa_{\ll}(T)\cdot \mfc_{\mu}(S)\cdot\left(1-\frac{x_1}{r_1}\right)\cdot \mfc_{\mu}(S)=\mfa_{\ll}(T)\cdot \mfc_{\mu}(S)\cdot\a_{\mu}\cdot\prod_{i=1}^{m}\left(1-\frac{x_i}{r_i}\right).
\end{equation}
\end{lemma}

We first prove a number of relations that will be useful throughout this section.

\begin{lemma}\label{lem:xixj}
 If $1\leq i\neq j\leq\mu_1$ are such that $\mu'_i\leq\mu'_j$ then
\begin{equation}\label{eq:zizj}
\begin{aligned}
&\mfc_{\mu}(S)\cdot z_i\cdot z_j = \mfc_{\mu}(S)\cdot z_i,\\
&\mfc_{\mu}(S)\cdot z_i^2 = \mfc_{\mu}(S)\cdot(\mu'_i-(\mu'_i-1)\cdot z_i).
\end{aligned}
\end{equation}
As a consequence, for $1\leq j<i\leq m$ we have
\begin{equation}\label{eq:xixj}
\begin{aligned}
&\mfc_{\mu}(S)\cdot x_i\cdot x_j =\mfc_{\mu}(S)\cdot x_i\cdot l_j,\\
&\mfc_{\mu}(S)\cdot x_i^2 =\mfc_{\mu}(S)\cdot ((l_i-h_i)\cdot x_i + l_i\cdot h_i).
\end{aligned}
\end{equation}
\end{lemma}

\begin{proof}
 We have for $1\leq i\neq j\leq\mu_1$ with $\mu'_i\leq\mu'_j$ that
\[
\begin{split}
\mfc_{\mu}(S)\cdot z_i\cdot z_j &= \mfc_{\mu}(S)\cdot\sum_{\substack{b\in C_i(S) \\ c\in C_j(S)}} (a,b)\cdot(a,c)= \mfc_{\mu}(S)\cdot\sum_{\substack{b\in C_i(S) \\ c\in C_j(S)}} (c,b)\cdot(a,b)\\
&= \mfc_{\mu}(S)\cdot\sum_{b\in C_i(S)}\left(\sum_{c\in C_j(S)} (c,b)\right)\cdot(a,b)\\
&\overset{(\ref{eq:Garnircol})}{=}\mfc_{\mu}(S)\cdot\sum_{b\in C_i(S)}(a,b)=\mfc_{\mu}(S)\cdot z_i.
\end{split}
\]
and (using the fact that $\mfc_{\mu}(S)\cdot(c,b)=-\mfc_{\mu}(S)$ if $b,c$ are in the same column of $S$)
\[
\begin{split}
\mfc_{\mu}(S)\cdot z_i^2 &= \mfc_{\mu}(S)\cdot\sum_{b,c\in C_i(S)} (a,b)\cdot(a,c)= \mfc_{\mu}(S)\cdot\left(\mu'_i+\sum_{b\neq c\in C_i(S)} (c,b)\cdot(a,b)\right)\\
&= \mfc_{\mu}(S)\cdot\left(\mu'_i-(\mu'_i-1)\cdot\sum_{b\in C_i(S)} (a,b)\right)=\mfc_{\mu}(S)\cdot(\mu'_i-(\mu'_i-1)\cdot z_i).\\
\end{split}
\]
To see how (\ref{eq:xixj}) follows from (\ref{eq:zizj}), write $\mc{C}_i=\{j:C_j(S)\subset B_i\}$ for the indices of the columns contained in the block $B_i$. We have
\[\mfc_{\mu}(S)\cdot x_i\cdot x_j=\mfc_{\mu}(S)\cdot \sum_{k\in\mc{C}_i,l\in\mc{C}_j}z_k\cdot z_l\overset{(\ref{eq:zizj})}{=}\mfc_{\mu}(S)\cdot \sum_{k\in\mc{C}_i,l\in\mc{C}_j}z_k=\mfc_{\mu}(S)\cdot l_j\cdot\sum_{k\in\mc{C}_i}z_k=\mfc_{\mu}(S)\cdot x_i\cdot l_j.\]
Using the fact that $\mu'_k=h_i$ if $k\in\mc{C}_i$ we get
\[
\begin{split}
 \mfc_{\mu}(S)\cdot x_i^2  =\ &\mfc_{\mu}(S)\cdot\left(\sum_{k\in\mc{C}_i} z_k^2 + \sum_{k\neq l\in\mc{C}_i}z_k\cdot z_l\right) \\
 \overset{(\ref{eq:zizj})}{=}&\mfc_{\mu}(S)\cdot\left(\sum_{k\in\mc{C}_i} (h_i-(h_i-1)\cdot z_k) + \sum_{k\in\mc{C}_i}z_k\cdot (l_i-1)\right) \\
 =\ &\mfc_{\mu}(S)\cdot\left(l_i\cdot h_i+\sum_{k\in\mc{C}_i} (l_i-h_i)\cdot z_k\right) = \mfc_{\mu}(S)\cdot ((l_i-h_i)\cdot x_i + l_i\cdot h_i).\qedhere
\end{split}
\]
\end{proof}

\begin{lemma}\label{lem:ccyclec}
 For $1\leq j_1<j_2<\cdots<j_k\leq\mu_1$, $k\geq 2$, and $b_{i}\in C_{j_i}(S)$ a collection of entries lying in distinct columns of $S$, we let $\s$ be the cyclic permutation $(a,b_{k},b_{k-1},\cdots,b_1)$. If $b_1,b_2$ lie in different rows of $S$ then $\mfc_{\mu}(S)\cdot\s\cdot\mfc_{\mu}(S)=0$. Otherwise, letting $\tau=(a,b_{k},b_{k-1},\cdots,b_2)$, we have $\mfc_{\mu}(S)\cdot\s\cdot\mfc_{\mu}(S)=\mfc_{\mu}(S)\cdot\t\cdot\mfc_{\mu}(S)$. As a consequence,
 \begin{equation}\label{eq:ccyclec}
  \mfc_{\mu}(S)\cdot\left(1-\frac{x_1}{r_1}\right)\cdot \mfc_{\mu}(S)=\mfc_{\mu}(S)\cdot\prod_{i=1}^{m}\left(1-\frac{x_i}{r_i}\right)\cdot \mfc_{\mu}(S).
 \end{equation}
\end{lemma}

\begin{proof}
 Suppose first that $b_1,b_2$ lie in distinct rows of $S$, say $b_2\in R_s(S)$, and let $x=S(s,j_1)$. Since $x\neq b_1$, $\s(x)=x$. Since $\mu'_{j_1}!\cdot\mfc_{\mu}(S)=\mfc_{\mu}(S)\cdot\mfb(C_{j_1}(S))$ and $\mu_{s}!\cdot\mfc_{\mu}(S)=\mfa(R_{s}(S))\cdot\mfc_{\mu}(S)$, it is enough to prove that
\[\mfb(C_{j_1}(S))\cdot\s\cdot\mfa(R_{s}(S))=\mfb(C_{j_1}(S))\cdot\mfa(\s(R_{s}(S)))\cdot\s=0.\]
The last equality holds true because the intersection $C_{j_1}(S)\cap\s(R_{s}(S))$ contains at least two elements, namely $x=\s(x)$ and $b_{1}=\s(b_{2})$.

Suppose now that $b_1,b_2$ belong to the same row of $S$. Since $(b_1,b_2)\in\mc{R}_S$, it follows that $(b_1,b_2)\cdot\mfc_{\mu}(S)=\mfc_{\mu}(S)$. Since $\s=\t\cdot(b_1,b_2)$, we obtain
\[\mfc_{\mu}(S)\cdot\s\cdot\mfc_{\mu}(S)=\mfc_{\mu}(S)\cdot\t\cdot(b_1,b_2)\cdot\mfc_{\mu}(S)=\mfc_{\mu}(S)\cdot\t\cdot\mfc_{\mu}(S).\]

To prove (\ref{eq:ccyclec}) note that it is equivalent, after subtracting the right hand side from the left and multiplying by $r_1$, to
\[\mfc_{\mu}(S)\cdot(r_1-x_1)\cdot\left(1-\prod_{i=2}^{m}\left(1-\frac{x_i}{r_i}\right)\right)\cdot\mfc_{\mu}(S)=0.\]
The left hand side expands into an expression (with coefficients $c_{\ul{b}}\in\bb{Q}$)
\[\mfc_{\mu}(S)\cdot\left(r_1-\sum_{b_1\in B_1}(a,b_1)\right)\cdot\left(\sum_{\substack{2\leq k\leq m\\  2\leq j_2<\cdots<j_k\leq m\\ b_{i}\in B_{j_i}}} c_{\ul{b}}\cdot(a,b_2)\cdot(a,b_{3})\cdots(a,b_{k})\right)\cdot\mfc_{\mu}(S).\]
Fix now $2\leq k\leq m$, $2\leq j_2<\cdots<j_k\leq m$, and $b_i\in B_{j_i}$ for $2\leq i\leq k$. Letting $j_1=1$, we have by the first part of the lemma that
\[\mfc_{\mu}(S)\cdot(a,b_1)\cdot(a,b_2)\cdot(a,b_{3})\cdots(a,b_{k})\cdot\mfc_{\mu}(S)=\mfc_{\mu}(S)\cdot(a,b_{k},\cdots,b_{3},b_2,b_1)\cdot\mfc_{\mu}(S)=0\]
if $b_1$ is not in the same row as $b_2$, and (with the previous notation for $\t$ and $\s$)
\begin{equation}\label{eq:cmutscmu}
\mfc_{\mu}(S)\cdot(1-(a,b_1))\cdot(a,b_2)\cdot(a,b_{3})\cdots(a,b_{k})\cdot\mfc_{\mu}(S)=\mfc_{\mu}(S)\cdot(\t-\s)\cdot\mfc_{\mu}(S)=0 
\end{equation}
if $b_1$ and $b_2$ are in the same row of $S$. Since there are exactly $r_1=l_1$ elements $b_1\in B_1$ lying in the same row as $b_2$, the conclusion follows by summing (\ref{eq:cmutscmu}) over all such $b_1$'s.
\end{proof}
 
\begin{proof}[Proof of Lemma~\ref{lem:betterprodconsec}] Both (\ref{eq:czv}) and (\ref{eq:czvc}) are trivially satisfied when $\mu'_v>\mu'_{v+1}$ because in this case $x_1=z_v$, $r_1=1$, $x_i=\tl{x}_{i-1}$ and $r_i=\tl{r}_{i-1}$ for $i>1$. We may then assume that $\mu'_v=\mu'_{v+1}$.


To prove (\ref{eq:czv}), since $\tl{x}_i=x_i$ and $\tl{r}_i=r_i$ for $i>1$, it is enough to show that
\begin{equation}\label{eq:czvshort}
 \mfc_{\mu}(S)\cdot(1-z_v)\cdot\left(1-\frac{\tl{x}_1}{\tl{r}_1}\right)=\mfc_{\mu}(S)\cdot\left(1-\frac{x_1}{r_1}\right).
\end{equation}
Since $\tl{x}_1=\sum_{j=v+1}^{v+l_1-1} z_j$, and $\mu'_v=\mu'_j$ for $v+1\leq j\leq v+l_1-1$, we get from (\ref{eq:zizj}) that 
\[\mfc_{\mu}(S)\cdot z_v\cdot\tl{x}_1=\mfc_{\mu}(S)\cdot z_v\cdot(l_1-1).\]
Using the relation above together with the fact that $\tl{r}_1=\tl{l}_1+1=l_1=r_1$, and $x_1=\tl{x}_1+z_v$,
\[\mfc_{\mu}(S)\cdot(1-z_v)\cdot\left(1-\frac{\tl{x}_1}{\tl{r}_1}\right)=\mfc_{\mu}(S)\cdot\left(1-z_v-\frac{\tl{x}_1}{l_1}+z_v\cdot\frac{l_1-1}{l_1}\right)=\mfc_{\mu}(S)\cdot\left(1-\frac{x_1}{r_1}\right),
\]
as desired. To prove (\ref{eq:czvc}), we multiply (\ref{eq:czvshort}) on the right by $\mfc_{\mu}(S)$, and it remains to show that
\[\mfc_{\mu}(S)\cdot(1-z_v)\cdot z_j\cdot\mfc_{\mu}(S)=0,\]
for $v+1\leq j\leq v+l_1-1$. We have
\[
\begin{split}
\mfc_{\mu}(S)\cdot z_v\cdot z_j\cdot\mfc_{\mu}(S)&=\sum_{\substack{b_1\in C_v(S)\\ b_2\in C_j(S)}}\mfc_{\mu}(S)\cdot (a,b_1)\cdot (a,b_2)\cdot\mfc_{\mu}(S)=\sum_{\substack{b_1\in C_v(S)\\ b_2\in C_j(S)}}\mfc_{\mu}(S)\cdot(a,b_2,b_1)\cdot\mfc_{\mu}(S)\\
&=\sum_{b_2\in C_j(S)}\mfc_{\mu}(S)\cdot(a,b_2)\cdot\mfc_{\mu}(S)=\mfc_{\mu}(S)\cdot z_j\cdot\mfc_{\mu}(S),
\end{split}
\]
where the second to last equality follows from Lemma~\ref{lem:ccyclec} and the fact that for each $b_2\in C_j(S)$ there exists exactly one $b_1\in C_v(S)$ situated in the same row.
\end{proof}

\begin{proof}[Proof of Theorem \ref{thm:yngconsec}] Let
\begin{equation}\label{eq:Pm}
 P_m=(x_1-r_1)\cdots(x_m-r_m).
\end{equation}
Using (\ref{eq:ccyclec}) we see that (\ref{eq:consec}) is equivalent, after multiplying by $(-1)^m r_1\cdots r_m$, to
\begin{equation}\label{eq:newconsec}
\mfa_{\ll}(T)\cdot \mfc_{\mu}(S)\cdot P_m\cdot \mfc_{\mu}(S)=\mfa_{\ll}(T)\cdot \mfc_{\mu}(S)\cdot\a_{\mu}\cdot P_m.
\end{equation}
Let
\begin{equation}\label{eq:sumrightcols}
X=\sum_{j=v}^{\mu_1} z_j=\sum_{i=1}^m x_i.
\end{equation}
We will show in Lemma \ref{lem:relations} that there exists a polynomial $Q(X)(=Q_m)$ with the property
\begin{equation}\label{eq:PmisQX}
\mfa_{\ll}(T)\cdot\mfc_{\mu}(S)\cdot P_m = \mfa_{\ll}(T)\cdot\mfc_{\mu}(S)\cdot Q(X).
\end{equation}
Using this, (\ref{eq:newconsec}) is equivalent to
\[\mfa_{\ll}(T)\cdot \mfc_{\mu}(S)\cdot Q(X)\cdot \mfc_{\mu}(S)=\mfa_{\ll}(T)\cdot \mfc_{\mu}(S)\cdot\a_{\mu}\cdot Q(X),\]
which is proved in Lemma \ref{lem:sumrightcomm} below.
\end{proof}
 
\begin{lemma}\label{lem:leftcol}
 If $z_j$ is as in (\ref{eq:colsum}) for some $j<v$, and $\s\in\S_n$ is a permutation that fixes all the entries of $S$ contained in its first $(v-1)$ columns, then 
\begin{equation}\label{eq:orthsigzj}
 \mfa_{\ll}(T)\cdot \mfc_{\mu}(S)\cdot\s\cdot (1-z_j)=0. 
\end{equation}
\end{lemma}

\begin{proof} If $\s(a)\neq a$, it follows that $\s(a)$ is contained in $C_i(S)$ for some $i$ with $\mu'_i\leq\mu'_j$. Since
\[\s\cdot(1-z_j)\cdot\s^{-1}=\left(1-\sum_{x\in C_j(S)}(\s(a),x)\right),\]
(\ref{eq:orthsigzj}) follows from the Garnir relation (\ref{eq:Garnircol}).

We may thus assume that $\s(a)=a$, so $\s\cdot(1-z_j)\cdot\s^{-1}=(1-z_j)$ and we are reduced to the case when $\s$ is the identity. Since $\mu'_j!\cdot \mfc_{\mu}(S)=\mfc_{\mu}(S)\cdot \mfb(C_j(S))$ and $\ll_{u}!\cdot \mfa_{\ll}(T)=\mfa_{\ll}(T)\cdot \mfa(R_u(T))$, it suffices to show that
\[\mfa(R_u(T))\cdot \mfc_{\mu}(S)\cdot \mfb(C_j(S))\cdot (1-z_j)\overset{(\ref{eq:symmskewsymmadda})}{=}\mfa(R_u(T))\cdot \mfc_{\mu}(S)\cdot \mfb(C_j(S)\cup\{a\})=0.\]
We show that in fact for any permutation $\t$ appearing in the expansion of $\mfc_{\mu}(S)$, i.e. $\t=\t_1\t_2$ with $\t_1\in\mc{R}_S$, $\t_2\in\mc{C}_S$, we have $\mfa(R_u(T))\cdot \t\cdot \mfb(C_j(S)\cup\{a\})=0$, or equivalently \[\mfa(R_u(T))\cdot \t\cdot \mfb(C_j(S)\cup\{a\})\cdot\t^{-1}=\mfa(R_u(T))\cdot \mfb(\t(C_j(S)\cup\{a\}))=0.\]
For this it is enough to prove that $|R_u(T)\cap\t(C_j(S)\cup\{a\})|\geq 2$. Since $\t(a)=a$ and $a\in R_u(T)$, it follows that $a$ is always an element of the intersection. Write $b$ for the $(u,j)$-entry of $S$, and observe that since $\t_2\in\mc{C}_S$, it induces a permutation of $C_j(S)$, so $b\in\t_2(C_j(S))$. Since $\t_1\in\mc{R}_S$ it follows that $\t_1(b)\in R_u(S)\subset R_u(T)$. It follows that the intersection $R_u(T)\cap\t(C_j(S)\cup\{a\})$ contains $\{a,\t_1(b)\}$, as desired.
\end{proof}

\begin{lemma}\label{lem:sumallcomm}
 Consider $\s\in \S_n$ with $\s(a)=a$, and for $z_j$ as in (\ref{eq:colsum}) let
\begin{equation}\label{eq:sumallcols}
Z=\sum_{j=1}^{\mu_1} z_j=\sum_{b\in S}(a,b). 
\end{equation}
We have $\s\cdot Z=Z\cdot\s$, and as a consequence $\mfc_{\mu}(S)\cdot Z=Z\cdot \mfc_{\mu}(S)$.
\end{lemma}

\begin{proof}
 This relation $\s\cdot Z=Z\cdot\s$ follows from the fact that
\[\s\cdot Z\cdot\s^{-1}=\sum_{b\in S} \s\cdot(a,b)\cdot\s^{-1}=\sum_{b\in S}(a,\s(b))=\sum_{b\in S}(a,b)=Z.\]
For the consequence, note that all $\s$ in the expansion of $\mfc_{\mu}(S)$ satisfy $\s(a)=a$.
\end{proof}

\begin{lemma}\label{lem:sumrightcomm}
 With $z_j$ as in (\ref{eq:colsum}), and $X$ as in (\ref{eq:sumrightcols}), we have that for any polynomial $P(X)$
\begin{equation}\label{eq:PXcomm}
\mfa_{\ll}(T)\cdot \mfc_{\mu}(S)\cdot\a_{\mu}\cdot P(X)=\mfa_{\ll}(T)\cdot \mfc_{\mu}(S)\cdot P(X)\cdot \mfc_{\mu}(S). 
\end{equation}
\end{lemma}

\begin{proof} By linearity, it suffices to prove (\ref{eq:PXcomm}) when $P(X)=X^t$. We argue by induction on $t$: for $t=0$, the result follows from (\ref{eq:idempyngsymm}). Assume now that $t>0$. Expanding $X^{t-1}$, we get a linear combination of permutations $\s$ that fix all the entries in the first $(v-1)$ columns of $S$. Using Lemma \ref{lem:leftcol} and letting $Z$ as in (\ref{eq:sumallcols}), we get 
\begin{equation}\label{eq:XtoZ}
\begin{aligned}
\mfa_{\ll}(T)\cdot \mfc_{\mu}(S)\cdot X^t &=\mfa_{\ll}(T)\cdot \mfc_{\mu}(S)\cdot X^{t-1}\cdot\left(X+\sum_{j=1}^{v-1}(z_j-1)\right)\\
&=\mfa_{\ll}(T)\cdot \mfc_{\mu}(S)\cdot X^{t-1}\cdot (Z-v+1).
\end{aligned}
\end{equation}
Multiplying both sides by $\a_{\mu}$ and using the induction hypothesis and the fact that $Z$ and $\mfc_{\mu}(S)$ commute (Lemma~\ref{lem:sumallcomm}), we get
\[
\begin{aligned}
\mfa_{\ll}(T)\cdot \mfc_{\mu}(S)\cdot\a_{\mu}\cdot X^t &=\mfa_{\ll}(T)\cdot \mfc_{\mu}(S)\cdot\a_{\mu}\cdot X^{t-1}\cdot(Z-v+1) \\
&=\mfa_{\ll}(T)\cdot \mfc_{\mu}(S)\cdot X^{t-1}\cdot \mfc_{\mu}(S)\cdot(Z-v+1) \\
&=\mfa_{\ll}(T)\cdot \mfc_{\mu}(S)\cdot X^{t-1}\cdot(Z-v+1)\cdot \mfc_{\mu}(S)
\end{aligned}
\]
Applying (\ref{eq:XtoZ}) again to the last term of the equality yields the desired conclusion.
\end{proof}

The last ingredient of the proof of Theorem \ref{thm:yngconsec} is (\ref{eq:PmisQX}), which we prove in the rest of this section. We define for $t=1,\cdots,m$ (recall the definition of $l_i,h_i$ and $r_i$ from (\ref{eq:hookri}))
\begin{equation}\label{eq:defsijQt}
\begin{aligned}
s_i^t=l_1+\cdots+l_i - h_{i+1}&\rm{ for } i=1,\cdots,t-1,\ s_t^t=l_1+\cdots+l_t,\\
X_t=x_1+x_2+\cdots+x_t,&\quad Q_t=(X_t-s_1^t)\cdot(X_t-s_2^t)\cdots(X_t-s_t^t),\\
P_t=(x_1-&r_1)\cdot(x_2-r_2)\cdots(x_t-r_t).
\end{aligned}
\end{equation}
Note that $X_m=X$ (as defined in (\ref{eq:sumrightcols})), that $Q_t$ is a polynomial in $X_t$, and that the definition of $P_m$ coincides with that in (\ref{eq:Pm}). For an element $f\in\K[\S_n]$, we define its \defi{right annihilator} by
\[RAnn(f)=\{p\in\K[\S_n]:f\cdot p=0\}.\]
We define a congruence relation $\con$ on $\K[\S_n]$ by
\begin{equation}\label{eq:defcon}
\begin{aligned}
f\con g &\Longleftrightarrow (f-g)\in I=\bigcap_{i\geq 0}RAnn(\mfa_{\ll}(T)\cdot c_{\mu}(S)\cdot X^i) \\
&\Longleftrightarrow\mfa_{\ll}(T)\cdot c_{\mu}(S)\cdot X^i\cdot f=\mfa_{\ll}(T)\cdot c_{\mu}(S)\cdot X^i\cdot g\ \forall\ i\geq 0.
\end{aligned}
\end{equation}
Since $I$ is a right ideal, $f\con g$ implies $fh\con gh$, but in general $hf\not\con hg$. However, for any polynomial $P(X)$ (where $X$ is as defined in (\ref{eq:sumrightcols})) we have
\begin{equation}\label{eq:PXfconPXg}
f\con g\Longrightarrow P(X)\cdot f\con P(X)\cdot g. 
\end{equation}
Moreover, we have
\begin{equation}\label{eq:conbyanncS}
 c_{\mu}(S)\cdot f=c_{\mu}(S)\cdot g\Longrightarrow f\con g.
\end{equation}
This is because (\ref{eq:defcon}) is equivalent after multiplying both sides by $\a_{\mu}$ and using (\ref{eq:PXcomm}) to
\[f\con g\Longleftrightarrow\mfa_{\ll}(T)\cdot c_{\mu}(S)\cdot X^i\cdot c_{\mu}(S)\cdot f=\mfa_{\ll}(T)\cdot c_{\mu}(S)\cdot X^i\cdot c_{\mu}(S)\cdot g\ \forall\ i\geq 0.\]
It follows from (\ref{eq:xixj}) and (\ref{eq:conbyanncS}) that
\begin{equation}\label{eq:conxixj}
\begin{aligned}
& x_i\cdot x_j \con x_i\cdot l_j,\\
& x_i^2 \con(l_i-h_i)\cdot x_i + l_i\cdot h_i.
\end{aligned}
\end{equation}

We will show that $P_m\con Q(X)$ for some polynomial $Q$ (namely $Q(X)=Q_m$), which will imply (\ref{eq:PmisQX}) and thus conclude the proof of Theorem \ref{thm:yngconsec}. We will prove by induction on $t$ the following

\begin{lemma}\label{lem:relations}
 With $\con$ as defined in (\ref{eq:defcon}), the following relations hold for $1\leq t\leq m$:
\begin{enumerate}
 \item[$(\rm{a}_t)$] $P_t\con Q_t$.
 \item[$(\rm{b}_t)$] $P_t\cdot(X_t+h_1) \con (X_t+h_1)\cdot P_t \con 0$.
\end{enumerate}
\end{lemma}

Before that, we formulate some preliminary results.

\begin{lemma}\label{lem:annPtQt}
For $2\leq j\leq m$ and $t\geq 1$ we have
\begin{equation}\label{eq:annPt}
\begin{aligned}
 &\mfc_{\mu}(S)\cdot x_j\cdot P_t = 0,\\
 &\mfc_{\mu}(S)\cdot (x_1+h_1)\cdot P_t = 0.
\end{aligned}
\end{equation}
 If $Q(X_t)$ is any polynomial in $X_t$ and $t+1\leq j\leq m$ then
\begin{equation}\label{eq:xjQXt}
 \mfc_{\mu}(S)\cdot x_j\cdot Q(X_t)= \mfc_{\mu}(S)\cdot x_j\cdot Q(s_t^t).
\end{equation}
In particular, for $t+1\leq j\leq m$ we have
\begin{equation}\label{eq:annQt}
 \mfc_{\mu}(S)\cdot x_j\cdot Q_t=0.
\end{equation}
\end{lemma}

\begin{proof} (\ref{eq:annPt}) follows from $\mfc_{\mu}(S)\cdot x_j\cdot (x_1-l_1)=0$ for $j\geq 2$, and $\mfc_{\mu}(S)\cdot (x_1+h_1)\cdot (x_1-l_1)=0$, which are special cases of (\ref{eq:xixj}).

Since $j>t$ we have by (\ref{eq:xixj}) that
\[\mfc_{\mu}(S)\cdot x_j\cdot X_t = \mfc_{\mu}(S)\cdot x_j\cdot(x_1+\cdots+x_t) = \mfc_{\mu}(S)\cdot x_j\cdot (l_1+\cdots+l_t) = \mfc_{\mu}(S)\cdot x_j\cdot s_t^t,\]
which when applied iteratively yields (\ref{eq:xjQXt}). (\ref{eq:annQt}) now follows from (\ref{eq:xjQXt}) and the fact that $Q_t$ is a polynomial in $X_t$ divisible by $(X_t-s_t^t)$ (see (\ref{eq:defsijQt})).
\end{proof}

\begin{lemma}\label{lem:annXPtQt}
For $2\leq j\leq m$ and $t\geq 1$ we have
\begin{equation}\label{eq:annXPt}
\begin{aligned}
 & x_j\cdot P_t \con 0,\\
 & (x_1+h_1)\cdot P_t \con 0.
\end{aligned}
\end{equation}
 If $Q(X_t)$ is any polynomial in $X_t$ and $t+1\leq j\leq m$ then
\begin{equation}\label{eq:conxjQXt}
 x_j\cdot Q(X_t) \con x_j\cdot Q(s_t^t).
\end{equation}
In particular, for $t+1\leq j\leq m$ we have
\begin{equation}\label{eq:annXQt}
 x_j\cdot Q_t\con 0.
\end{equation}
\end{lemma}
 
\begin{proof} This follows from Lemma \ref{lem:annPtQt} using (\ref{eq:conbyanncS}).
\end{proof}

\begin{lemma}\label{lem:QXtisQXj}
 If $Q$ is a polynomial then
\begin{equation}\label{eq:QXtisQXj}
Q(X_t)\cdot (X_t-s_t^t)\con Q(X)\cdot (X_t-s_t^t).
\end{equation}
\end{lemma}

\begin{proof} We prove by induction on $i$ that $X^i\cdot(X_t-s_t^t)\con X_t^i\cdot(X_t-s_t^t)$, the result being trivial for $i=0$. Assuming that the result is true for some $i$ and multiplying both sides by $X$, we get using (\ref{eq:PXfconPXg}) that
\[X^{i+1}\cdot(X_t-s_t^t)\con X\cdot X_t^i\cdot(X_t-s_t^t) = X_t^{i+1}\cdot(X_t-s_t^t) + \sum_{j=t+1}^m x_j\cdot X_t^i\cdot(X_t-s_t^t)\overset{(\ref{eq:conxjQXt})}{\con}X_t^{i+1}\cdot(X_t-s_t^t).\]
\end{proof}

\begin{proof}[Proof of Lemma \ref{lem:relations}]
 We do induction on $t$, showing that $(\rm{a}_t)\Rightarrow (\rm{b}_t)$ and $(\rm{a}_t),(\rm{b}_t)\Rightarrow(\rm{a}_{t+1})$. When $t=1$, we have $r_1=s_1^1=l_1$ so $P_1=Q_1=x_1-l_1$ and $(\rm{a}_1)$ holds.

\noindent$\ul{(\rm{a}_t)\Rightarrow (\rm{b}_t)}$: We have
\[P_t\cdot(X_t+h_1)\overset{(\rm{a}_t)}{\con} Q_t\cdot(X_t+h_1)=(X_t+h_1)\cdot Q_t\overset{(\ref{eq:annXQt})}{\con}(X+h_1)\cdot Q_t\overset{(\ref{eq:PXfconPXg})}{\underset{(\rm{a}_t)}{\con}}(X+h_1)\cdot P_t\overset{(\ref{eq:annXPt})}{\con}0.\]

\noindent$\ul{(\rm{a}_t),(\rm{b}_t)\Rightarrow (\rm{a}_{t+1})}$: We have
\begin{equation}\label{eq:redPtplus1}
\begin{aligned}
 P_{t+1}=P_t\cdot(x_{t+1}-r_{t+1}) &\overset{(\rm{b}_t)}{\con}P_t\cdot(x_{t+1}-r_{t+1}+X_t+h_1) \\
& \overset{(\ref{eq:hookri}),(\ref{eq:defsijQt})}{=} P_t\cdot(X_{t+1}+h_{t+1}-s_{t+1}^{t+1}) \overset{(\rm{a}_t)}{\con}Q_t\cdot(X_{t+1}+h_{t+1}-s_{t+1}^{t+1}).
\end{aligned}
\end{equation}
We have $Q_t=Q(X_t)\cdot(X_t-s_t^t)$, where $Q$ is the polynomial $Q(z)=\prod_{i=1}^{t-1}(z-s^t_i)$. By (\ref{eq:QXtisQXj}), $Q_t\con Q(X)\cdot(X_t-s_t^t)$, so (\ref{eq:redPtplus1}) becomes
\begin{equation}\label{eq:redPtmore}
P_{t+1}\con Q(X)\cdot(X_t-s_t^t)\cdot(X_{t+1}+h_{t+1}-s_{t+1}^{t+1}). 
\end{equation}
We have
\[
\begin{split}
 &(X-s^{t+1}_t)\cdot(X_{t+1}-s^{t+1}_{t+1}) \overset{(\ref{eq:QXtisQXj})}{\con} (X_{t+1}-s^{t+1}_t)\cdot(X_{t+1}-s^{t+1}_{t+1}) \\
= &(X_t-s_t^t)\cdot(X_{t+1}-s^{t+1}_{t+1})+(x_{t+1}+h_{t+1})\cdot(X_{t+1}-s^{t+1}_{t+1})\\
= &(X_t-s_t^t)\cdot(X_{t+1}-s^{t+1}_{t+1})+(x_{t+1}+h_{t+1})\cdot(X_t-s_t^t)+(x_{t+1}+h_{t+1})\cdot(x_{t+1}-l_{t+1})\\
 \overset{(\ref{eq:conxixj})}{\underset{(\ref{eq:conxjQXt})}{\con}} &(X_t-s_t^t)\cdot(X_{t+1}-s^{t+1}_{t+1})+h_{t+1}\cdot(X_t-s_t^t)= (X_t-s_t^t)\cdot(X_{t+1}+h_{t+1}-s_{t+1}^{t+1}).\\
\end{split}
\]
By (\ref{eq:PXfconPXg}) this chain of congruences is preserved if we multiply on the left by $Q(X)$, so (\ref{eq:redPtmore}) is equivalent to
\begin{equation}\label{eq:reducedPt}
P_{t+1}\con Q(X)\cdot(X-s^{t+1}_t)\cdot(X_{t+1}-s^{t+1}_{t+1}).
\end{equation}
Using (\ref{eq:QXtisQXj}) with $t$ replaced by $(t+1)$ and $Q(X)$ replaced by $Q(X)\cdot(X-s^{t+1}_t)$, and using the fact that $s_i^t=s_i^{t+1}$ for $i\leq t-1$, we obtain
\[P_{t+1}\con Q(X_{t+1})\cdot(X_{t+1}-s^{t+1}_t)\cdot(X_{t+1}-s^{t+1}_{t+1})=Q_{t+1},\]
which concludes the proof of Lemma \ref{lem:relations} and that of Theorem \ref{thm:yngconsec}.
\end{proof}

\section{Proof of Theorem \ref{thm:main}}\label{sec:proofThm1}

We prove Theorem \ref{thm:main} by induction on the difference $(n-k)$. When $n-k=0$, the theorem follows from the quasi-idempotence of Young symmetrizers (\ref{eq:idempyngsymm}). When $n-k=1$, the theorem is a consequence of Theorem \ref{thm:yngconsec} (Remark \ref{rem:prodconsec}). We may thus assume that $n-k\geq 2$.

We consider the subtableau $U$ of $T$ obtained by removing the rightmost corner box of $T$ not contained in $S$. More precisely, $U$ is obtained by removing the box in position $(u,v)$ where $v=\max\{j:\ll'_j\neq\mu'_j\}$ and $u=\ll'_v$. We write $\delta$ for the shape of $U$. For example take $n=8$, $k=4$, $\ll=(3,2,2,1)$, $\mu=(3,1)$, in which case $(u,v)=(3,2)$ and $\delta=(3,2,1,1)$:
\[\Yvcentermath1 T=\young(123,45,67,8),\quad S=\young(123,4),\quad U=\young(123,45,6,8),\]

We have by induction
\begin{equation}\label{eq:cdcm}
 \mfc_{\delta}(U)\cdot \mfc_{\mu}(S)=\mfc_{\delta}(U)\cdot\left(\sum_{\s\in L(U;S)}m_{\s}\cdot \s\right)
\end{equation}
with $m_{\1}=\a_{\mu}$. By Theorem \ref{thm:yngconsec}
\begin{equation}\label{eq:clcd}
 \mfc_{\ll}(T)\cdot \mfc_{\delta}(U) = \mfc_{\ll}(T)\cdot\left(\sum_{\t\in L(T;U)}n_{\t}\cdot \t\right)
\end{equation}
where $n_{\1}=\a_{\delta}$, and moreover $n_{\t}\neq 0$ only for permutations $\t$ that fix the entries in the first $v$ columns of $U$. By the choice of $(u,v)$, all entries of $U$ in the columns $v+1,v+2,\cdots$ belong to $S$, so the permutations $\t$ appearing in (\ref{eq:clcd}) fix all the entries of $U$ outside $S$. Combining (\ref{eq:cdcm}) with (\ref{eq:clcd}) we obtain
\begin{equation}\label{eq:newclcdcm}
\mfc_{\ll}(T)\cdot \mfc_{\delta}(U)\cdot \mfc_{\mu}(S)=\mfc_{\ll}(T)\cdot\left(\sum_{\t\in L(T;U)}n_{\t}\cdot \t\right)\cdot\left(\sum_{\s\in L(U;S)}m_{\s}\cdot \s\right), 
\end{equation}
It is clear that if $\t\in L(T;U)$ and $\s\in L(U;S)$ then for every $s\in S$ either $(\t\cdot\s)(s)=s$ or $(\t\cdot\s)(s)$ lies strictly to the left of $s$. Assume now that $(\t\cdot\s)(s)=s$ for all entries $s$ in $S$. Since every $\t$ appearing with non-zero coefficient in (\ref{eq:clcd}) fixes the entries of $U$ outside $S$, it must be that $\s$ permutes the entries of $S$, but then the definition of $L(U;S)$ yields $\s=\1$. It follows that $\t$ fixes the entries of $S$ along with those of $U\setminus S$, so $\t=\1$.

We claim that
\begin{equation}\label{eq:clcdcm}
 \mfc_{\ll}(T)\cdot \mfc_{\mu}(S)=\frac{1}{\a_{\delta}}\cdot \mfc_{\ll}(T)\cdot \mfc_{\delta}(U)\cdot \mfc_{\mu}(S). 
\end{equation}
To prove (\ref{eq:clcdcm}) it suffices to show that for every $\t\neq\1$ appearing in (\ref{eq:clcd}) with non-zero coefficient we have $\mfc_{\ll}(T)\cdot\t\cdot \mfc_{\mu}(S)=0$, or equivalently (after multiplying on the right by $\t^{-1}$) $\mfc_{\ll}(T)\cdot\mfc_{\mu}(\t\cdot S)=0$. Since $\t\neq\1$, there exists an entry $s$ in $S$ such that $\t(s)=a$ is the unique entry of $T$ outside $U$ ($s=b_1$ in the notation of Remark \ref{rem:prodconsec}). If we write $(i,j)$ for the coordinates of $s$, we have $j>v$. Let $s'=S(i,v)$ and note that $\t(s')=s'$. We have that $a$ and $s'$ are both contained in $C_v(T)$ and in $R_i(\t\cdot S)$, which forces $\mfc_{\ll}(T)\cdot\mfc_{\mu}(\t\cdot S)=0$, as desired. The conclusion of Theorem~\ref{thm:main} now follows by expanding~(\ref{eq:newclcdcm}) and using~(\ref{eq:clcdcm}).

\section*{Acknowledgments} 
I would like to thank Steven Sam and John Stembridge for kindly answering my questions about the project, and the anonymous referees for many helpful suggestions. Experiments with the computer algebra softwares Macaulay2 \cite{M2} and Sage \cite{sage} have provided numerous valuable insights.


\end{document}